\documentclass[12pt]{article}
\usepackage[utf8x]{inputenc}
\usepackage{enumerate}
\usepackage{amssymb,amsmath,amsthm,amscd,latexsym}
\usepackage{times}
\usepackage{setspace}
\usepackage{graphics, graphicx, color}
\def\titlerunning#1{\gdef\titrun{#1}}
\makeatletter
\def\author#1{\gdef\autrun{\def\and{\unskip, }#1}\gdef\@author{#1}}
\def\address#1{{\def\and {\\\hspace*{18pt}} \renewcommand{\thefootnote}{}%
\footnote {#1}}%
\markboth{\autrun}{\titrun}}
\makeatother
\def\email#1{e-mail: #1}
\def\subjclass#1{{\renewcommand{\thefootnote}{}%
\footnote{\emph{Mathematics Subject Classification (2010):} #1}}}
\def\keywords#1{\par\medskip
\noindent\textbf{Keywords.} #1}

\newtheorem{theo}{Theorem}[section]

\newtheorem{prop}[theo]{Proposition}

\newtheorem{lem}[theo]{Lemma}

\newtheorem{defi}[theo]{Definition}

\theoremstyle{definition}
\newtheorem{rem}[theo]{Remark}
\newtheorem{example}{Example}
\newtheorem{claim}{Claim}

\newtheorem*{cons}{Construction}

\newcommand{\NN}{\mathbf{N}}
\newcommand{\Int}{\mathrm{Int}}

\newcommand{\RR}{\mathbf{R}}
\newcommand{\ZZ}{\mathbf{Z}}

\newcommand{\cF}{\mathcal{F}}
\newcommand{\cS}{\mathcal{S}}

\newcommand{\cC}{\mathcal{C}}
\newcommand{\cX}{\mathcal{X}}
\newcommand{\cP}{\mathcal{P}}

\let\bord\partial
\let\bydef\emph

\begin{document}
\begin{spacing}{1.2}

\titlerunning{Open $3$-manifolds which are connected sums of closed ones}
\title{Open $3$-manifolds which are connected sums\\ of closed ones}

\author{Laurent Bessi\`eres \and G\'erard Besson \and Sylvain Maillot}

\maketitle

\address{L. Bessi\`eres: Institut de Math\'ematiques de Bordeaux, Universit\'e de Bordeaux, 33405 Talence, FRANCE; \email{laurent.bessieres@math.u-bordeaux.fr} 
\and 
G. Besson: Institut Fourier, Universit\'e Grenoble Alpes, 100 rue des maths, 38610 Gi\`eres, France; 
\email{g.besson@univ-grenoble-alpes.fr} 
\and 
S. Maillot: Institut Montpelli\'erain Alexander Grothendieck, CNRS, Universit\'e de Montpellier;
\email{sylvain.maillot@univ-montp2.fr}}


\begin{abstract}
We consider open, oriented $3$-manifolds which are infinite connected sums of closed $3$-manifolds. We introduce some topological invariants for these manifolds and obtain a classification in the case where there are only finitely many summands up to diffeomorphism. This result encompasses both the Kneser-Milnor Prime Decomposition Theorem for closed $3$-manifolds and the Kerékj\'art\'o-Ri\-chards classification theorem for open surfaces.
\end{abstract}

\section{Introduction}\label{sec:intro}
In~\cite{scott:noncompact}, P.~Scott studied decompositions of open $3$-manifolds as (possibly infinite) connected sums of prime manifolds which may or may not be compact. He gave examples of open $3$-manifolds that do not have such decompositions. By contrast, in~\cite{B2M:scalar} the authors proved that open $3$-manifolds which carry certain Riemannian metrics of positive scalar curvature do admit such decompositions, and in addition all summands are compact. In order to recall the result, we introduce some terminology.

Throughout the paper, we work in the category of smooth, oriented 3-manifolds. We say that two oriented 3-manifolds $M_1,M_2$ are \bydef{isomorphic} if there exists an orientation-preserving diffeomorphism from $M_1$ to $M_2$. For convenience we say that a 3-manifold $N$ has \bydef{spherical boundary} if every component of $\bord N$ (if any) is a 2-sphere. Then there is a unique manifold, up to isomorphism,  obtained by gluing a $3$-ball to every boundary component of $N$. This operation is called \bydef{capping-off}. The capped-off manifold is denoted by $\widehat N$.

\begin{defi}
Let $\mathcal{X}$ be a class of connected oriented $3$-manifolds without boundary. An oriented $3$-manifold $M$ is said to be \bydef{decomposable over} $\mathcal{X}$ if $M$ contains a locally finite collection of pairwise disjoint embedded $2$-spheres $\{S_i\}$ such that for every connected component $C$ of $M$ split along $\{S_i\}$, the capped-off manifold $\widehat C$ is isomorphic to some member of $\mathcal{X}$ or to $S^3$.
\end{defi}

The main result of~ \cite{B2M:scalar} can be formulated as follows: if $M$ is an open $3$-manifold which admits a complete Riemannian metric of bounded geometry and uniformly positive scalar curvature, then there exists a finite familly $\cF$ of spherical manifolds such that $M$ is decomposable over $\mathcal{F} \cup\{S^2\times S^1\}$. This begs the question of whether such manifolds, and more generally
$3$-manifolds which are decomposable over some finite collection of closed $3$-manifolds, can be classified. In this paper we show that this is the case, using a set of invariants inspired by the work of Kerékj\'art\'o~\cite{Ker} and Richards~\cite{Rich} to classify open surfaces:\footnote{While the Kerékj\'art\'o-Richards theorem also applies to nonorientable surfaces, in this paper we restrict attention to orientable manifolds for simplicity.}  the space of ends of $M$  together with a colouring depending on whether a given closed prime $3$-manifold appears infinitely many times in any neighbourhood of this end or not. 

Recall that a closed $3$-manifold $K$ is \bydef{prime} if it is connected, not diffeomorphic to $S^3$, and whenever $K$ is diffeomorphic to $M_1\# M_2$, one of the $M_i$'s is diffeomorphic to $S^3$. In this paper we shall use the word `prime' only for \emph{closed} manifolds. The class of oriented prime $3$-manifolds will be denoted by $\cP$.

Let $K$ be a (possibly disconnected) compact oriented $3$-manifold with spherical boundary. For each $P\in\cP$, we denote by  $n_P(K)$ the sum over all connected components $L$ of $K$ of the number of summands isomorphic to $P$ in the Kneser-Milnor decomposition of the capped-off manifold $\widehat L$. Thus the Kneser-Milnor theorem implies: every closed oriented $3$-manifold is decomposable over some finite collection of prime oriented $3$-manifolds, and two closed, connected, oriented $3$-manifolds $K_1,K_2$ are isomorphic if and only if $n_P(K_1)=n_P(K_2)$ for every $P\in\cP$.

We now come to the definition of our invariants.
Let $P$ be an oriented prime 3-manifold. Let $U$ be an open oriented $3$-manifold.Then we define  $n_P(U)\in\NN\cup\{\infty\}$ as the supremum of the numbers $n_P(K)$ where $K$ ranges over all compact submanifolds of $U$ with spherical boundary. Let $M$ be an open connected oriented $3$-manifold. An end $e$ of $M$ has \bydef{colour} $P$ if for every neighbourhood $U$ of $e$, with compact boundary, the number $n_P(U)$ is nonzero. (Then it is actually infinite, see below.) The set of ends of $M$ of colour $P$ is denoted by $E_P(M)$. By Lemma~\ref{lem:closed} below, it is a closed subset of $E(M)$.
Note that an end may have several colours, i.e. $E_P(M)\cap E_{P'}(M)$ can be nonempty for $P,P'\in\cP$. In fact, the closed subsets $E_P(M)$ can be quite arbitrary by our realisation theorem, Theorem~\ref{theo:construction} below.

Let $M,M'$ be open connected oriented $3$-manifolds. A map $\phi:E(M) \to E(M')$ is called \bydef{colour-preserving} if for every $P\in\cP$, we have $\phi(E_P(M))=E_P(M')$.

\begin{theo}[Main theorem]\label{theo:main}
Let $\mathcal{X}$ be a finite collection of prime oriented $3$-manifolds. Let  $M,M'$ be open connected oriented $3$-manifolds which are decomposable over $\mathcal{X}$. Then $M$ is isomorphic to $M'$ if and only if the following conditions hold:
  \begin{enumerate}[\hspace{10pt}(1)]
  \item For every $P\in\cP$, the numbers $n_P(M)$ and $n_P(M')$ are equal.
    \item There is a colour-preserving homeomorphism $\phi : E(M) \to E(M')$.
  \end{enumerate}
\end{theo}

\begin{rem}
\begin{enumerate}
\item By Thereom \ref{thm:tree} and Proposition~\ref{prop:invariants} below, it suffices in Condition~(1)  to assume this for all $P\in \cX \cup \{S^2\times S^1\}$, since all the other numbers vanish. 
\item Example~\ref{example:one} below shows that the finiteness assumption on $\cX$ is necessary.
\item In the case where $n_P(M)<\infty$ for all $P \in \cP$, the manifold is determined, up to isomorphism,  by $E(M)$ and the numbers $n_P(M)$. 
\end{enumerate}
\end{rem}

To conclude this introduction, let us mention that the notion of infinite connected sum is also useful in large dimensions, see e.g.~\cite{whyte}.
 
The paper is organized as follows: in Section~\ref{sec:connected sums}, we recall, in a slight more precise form, the notion of connected sum of closed manifolds along a possibly infinite graph, discussed in~ \cite{B2M:scalar, B2M2:space_of_metrics}, and explain the relationship between this notion and decomposability. We state Theorem~\ref{thm:tree}, which allows to turn the graph into a tree, possibly at the expense of adding $S^2\times S^1$ summands. We also state Proposition~\ref{prop:invariants} and Theorem~\ref{theo:construction}, and give Example~\ref{example:one} mentioned above. In Section~\ref{sec:preliminaries}, we discuss the properties of special exhaustions of decomposable $3$-manifolds called \emph{spherical exhaustions}, which play an important role in all the subsequent proofs. Proposition~\ref{prop:invariants} and Theorem~\ref{theo:construction} are proved there. Section~\ref{sec:proof of 2.3} is devoted to the proof of Theorem~\ref{thm:tree}. Section~\ref{sec:proof of main} is dedicated to the proof of Theorem~\ref{theo:main}. Some concluding remarks are gathered in Section~\ref{sec:conclusion}. We finish with an appendix reviewing the theory of spaces of ends.\\

Acknowledgment: this work began during ANR project GTO ANR-12-BS01-0004 and completed during ANR project CCEM ANR-17-CE40-0034.

\section{Connected sums along graphs}\label{sec:connected sums}
In this section we reformulate our notion of decomposable manifolds in terms of performing connected sums of possibly infinitely many closed oriented manifolds along some locally finite graph. This was discussed in previous articles of the authors and F.~C.~Marques~\cite{B2M:scalar, B2M2:space_of_metrics} (cf.~\cite{scott:noncompact}.) 
Here though, we need to be more precise regarding orientations, and fix some notation for future use.

Let $I$ be a subset of $\NN$ and $\mathcal X=\{X_k\}_{k\in I}$ be a family of closed connected oriented $3$-manifolds. By convention we assume henceforth that $0\in I$ and $X_0=S^3$. A \bydef{coloured graph} is a pair $(G,f)$ where $G$ is a locally finite connected graph, possibly with loops, with vertex set $V(G)$ and $f$ is a map from $V(G)$ to $I$.

\begin{cons}
For each vertex $v\in V(G)$, let $Y_v$ be a copy of $X_{f(v)}$ with $d(v)$ disjoint open 3-balls removed, where $d(v)$ is the degree of $v$. Thus $Y_v$ is a compact manifold with spherical boundary. Let $Y$ be the disjoint union of all $Y_v$'s. Then glue the $Y_v$'s to each other along the edges of $G$, using an orientation-reversing diffeomorphism. We denote by $M(G,f)$ the resulting oriented $3$-manifold; it is well-defined up to isomorphism (see \cite{B2M2:space_of_metrics}, Section 2.1). For each edge $a$ of $G$ there is an embedded $2$-sphere $S_a$ in $M(G, f)$ along which the gluing has been done.
\end{cons}

\begin{defi}
Let $\cX$, $G$ and $f$ be as above. The oriented manifold $M(G,f)$ is called the \bydef{manifold obtained by connected sum along} the coloured graph $(G,f)$.
\end{defi}

\begin{rem}\label{rem:M(G,f)}
\begin{enumerate}
\item The manifold $M(G,f)$ is  connected and without boundary. It is closed if and only if $G$ is finite.
\item If there is a loop in $G$ at the vertex $v$, then the natural map from $Y_v$ to $M(G,f)$ is not injective. If there is none, then $Y_v$ can be identified with its image in $M$.
\item The family $\{S_a\}$ is a locally finite family of embedded, pairwise disjoint $2$-spheres in $M(G,f)$ whose dual graph is isomorphic to $G$. It follows that $M(G,f)$ is decomposable over $\cX$.
Conversely, if $M$ is decomposable over $\cX$, then $M$ is isomorphic to the connected sum along the dual graph $G$ of the splitting collection of $2$-spheres with appropriate colouring. 
\end{enumerate}
\end{rem}

In the compact case this definition is slightly nonstandard in the sense that in the usual definition the graph is a \emph{tree}. It is well-known, however, that a connected sum along a finite graph 
 can be made into a connected sum along a tree at the expense of adding vertices with $S^2\times S^1$ factors. This is also true in the infinite case:
\begin{theo}\label{thm:tree}
Let $\cX=\{X_k\}$ be a (possibly infinite) family of closed connected oriented manifolds with $X_0=S^3$. Let $M$ be a connected open oriented $3$-manifold. If $M$ is decomposable over $\mathcal X$, then it is isomorphic 
to a connected sum along a coloured tree of members of $\cX \cup \{S^2 \times S^1\}$. 
\end{theo}

Our next task is to define combinatorial invariants of coloured trees and explain the relationship between these invariants and the topological invariants of decomposable 3-manifolds introduced earlier. In particular, Proposition~\ref{prop:invariants} below shows that the latter can be effectively calculated.

Let $\cX=\{P_k\}_{k\in I}$ be a family of pairwise non-isomorphic closed oriented manifolds such that $X_0=S^3$ and $P_k$ is a prime manifold for all $k>0$. Let $(T,f)$ be a coloured tree and let $M=M(T,f)$.
For every $k\in I\setminus\{0\}$, we denote by $n_k(T,f)\in\textbf{N}\cup \{\infty\}$ the number of vertices $v$ of $T$ such that $f(v)=k$. For every pre-end $e=U_1\supset U_2\supset \cdots$ of $T$ we let $n_k(e)$ be the limit as $n$ goes to $\infty$ of the number of vertices $v$ in $U_n$ such that $f(v)=k$. This passes to a well-defined invariant $n_k(e^*)$ for ends of $T$. We let $E_k(T,f)$ denote the set of ends $e^*$ of $T$ such that $n_k(e^*)=\infty$.

\begin{prop}\label{prop:invariants}
Let $\mathcal X$ be as above. Let $(T,f)$ be a coloured tree and $M=M(T,f)$. Then the following assertions hold:
\begin{enumerate}
\item For every $k\in I\setminus\{0\}$ we have $n_{P_k}(M)=n_k(T,f)$. 
\item There is a homeomorphism $\phi:E(M)\to E(T)$ such that for every $k\in I\setminus\{0\}$, we have $E_k(T,f)=\phi(E_{P_k}(M))$.
\item For every prime oriented manifold $P$ which is not isomorphic to a member of $\cX$, we have $n_P(M)=0$ and $E_P(M)=\emptyset$.
\end{enumerate}
\end{prop}

\begin{example} \label{example:one}
Let $(P_n)_{n \in \NN}$ be a sequence of pairwise non isomorphic oriented $3$-manifolds, such that $P_0=S^3$ and $P_n$ is prime for all $n>0$. Let $Z$ be the Cayley graph of $\ZZ$ with standard generators. We define two colourings $f_1,f_2$ on $G$ by setting $f_1(n)=0$ if $n<0$,  $f_1(n)=n$ if $n\ge 0$, $f_2(n)=-2n$ if $n<0$ and $f_2(n)=2n+1$ if $n\ge 0$.

Then for every $i\in \NN\setminus \{0\}$ we have $n_i(Z,f_1)=n_i(Z,f_2)=1$, and the subsets $E_i(Z,f_j)$ are empty for all $i>0$, $j\in\{1,2\}$. It follows from Proposition~\ref{prop:invariants} that the manifolds $M(Z,f_1)$ and $M(Z,f_2)$ have the same invariants. However, $M_1$ and $M_2$ are not isomorphic. Indeed, $M_1$ has an end that has a neighbourhood diffeomorphic to $S^2\times\RR$, while no end of $M_2$ has this property.
\end{example} 

Finally, we state our realisation theorem.
\begin{theo}\label{theo:construction}
Let $E$ be a compact, metrisable, totally disconnected space. Consider a collection $E_1, \dots , E_k$ of  closed subsets of $E$. Let $n_1,\ldots,n_\ell$ be nonnegative integers. Let $\cX=\{P_1, \ldots, P_k, Q_1,\ldots,Q_\ell\}$ be a family of prime manifolds. Then there exists a connected oriented $3$-manifold $M$ which is a connected sum along a tree of members of $\cX \cup \{S^3\}$ with the following properties:
\begin{itemize}
\item For every $i \in \{1,\ldots,\ell\}$ we have $n_{M}(Q_i)=n_i$.
\item There exists a homeomorphism $\phi : E(M) \to E$ such that $\phi(E_{P_i}(M))=E_i$ for every $i \in \{1,\ldots,k\}$. 
\end{itemize}
\end{theo}

\section{Preliminaries}\label{sec:preliminaries}
 
 \subsection{Spherical exhaustions}
Recall that an \bydef{exhaustion} of an open manifold $M$ is a sequence $(K_n)$ of compact connected submanifolds of $M$ such that $M=\bigcup_n K_n$, and $\overset{\circ}{K}_{n+1}\supset K_n$ for every $n$. We may assume that for each $n$, every
component of $M \setminus K_n$ is unbounded, i.e.~not contained in any compact subset. 
A \bydef{spherical exhaustion} is an exhaustion  $(K_n)$ where each $K_n$ has spherical boundary. 

Let $P\in\cP$. Let $M$ be an open connected oriented $3$-manifold. Let $e=U_1\supset U_2 \supset \dots\,$ be a pre-end of $M$. Then the sequence $(n_P(U_i))$ is non increasing. Let $n_P(e)$ denote the limit of this sequence. 

\begin{lem}
The number $n_P(e)$ is either 0 or $+\infty$.
\end{lem}

\begin{proof}
If $n_P(U_i)=\infty$ for every $i$, then $n_P(e)=\infty$. Otherwise, there exists $i$ such that $n_P(U_i)$ is finite. Thus there exists a compact submanifold $K\subset U_i$ with spherical boundary such that $n_P(K)=n_P(U_i)$. By definition of a pre-end, there exists $j>i$ such that $K\cap U_j=\emptyset$. Then $n_P(U_j)=0$, otherwise $U_j$ would contain some compact $K'$ with nonzero $n_P(K')$ and we would have $n_P(K\cup K')=n_P(K) + n_{P}(K')>n_P(K)$, contradicting the choice of $K$.
\end{proof}

Furthermore, if $e,e'$ are equivalent pre-ends, then it is readily checked that $n_P(e)=n_P(e')$. Thus we have a well-defined invariant $n_P(e^*)\in \{0,\infty\}$ for each end $e^*$ of $M$. We remark that the previously defined set of ends of $M$ of colour $P$  is the set of $e^*\in E(M)$ such that $n_P(e^*)=\infty$.

\begin{lem}\label{lem:closed}
The set $E_P(M)$ is a closed subset of $E(M)$. 
\end{lem}

\begin{proof}
Let $e^\ast \in E(M) \setminus E_P(M)$ and let $e=U_1 \supset U_2 \supset \ldots$ representing $e^\ast$. Then  $n_P(U_i)=0$ for all $i$ large enough, which implies that all ends $e'^\ast \in U_i^\ast$ satisfy $n_P(e'^\ast)=0$. Therefore $U_i^\ast$ is an open neighbourhood of $e^\ast$ disjoint from $E_P(M)$, proving that $E(M) \setminus E_P(M)$ is open.
\end{proof}

\subsection{Proof of Proposition~\ref{prop:invariants}}
 \label{subsec:prop:invariant}
Let $\cX=\{P_k\}_{k\in I}$ be a family of pairwise non-isomorphic closed oriented manifolds such that $X_0=S^3$ and $P_k$ is prime for all $k>0$. Let $(T,f)$ be a coloured tree and let $M=M(T,f)$. Recall from Remark \ref{rem:M(G,f)} that for every vertex $v$ in $T$ we have a submanifold $Y_v$ of $M$ and for every edge $a$ of $T$ there is an embedded $2$-sphere $S_a$ in $M$.

We construct a proper continuous map $F:M\longrightarrow T$ which naturally induces a homeomorphism $\phi: E(M)\longrightarrow E(T)$. We recall that $T$ is endowed with the topology given by the length distance for which the edges are isometric to the interval $[0, 1]$. For each edge $a$ of $T$, let $Z_a\subset M$ be a tubular neighbourhood of $S_a$, whose size is chosen so that $Z_a\cap Z_{a'}=\emptyset$ if $a\ne a'$. We also choose a parametrisation $Z_a\simeq S_a\times (-1,1)$ with $S_a\simeq S_a\times \{0\}$. For each vertex $v\in T$ we define $C_v=Y_v\setminus\bigcup_aZ_a$, which we call the core of $Y_v$. Notice that $C_v$ is homeomorphic to $Y_v$.  We then define $F:M\longrightarrow T$ by
\begin{itemize}
  \item[1.] $F(C_v)=v$ for each vertex $v\in T$.
  \item[2.] $F(Z_a)=a$ for each edge $a\in T$. More precisely, $$F: Z_a\simeq S_a\times (-1,1)\longrightarrow (-1,1) \simeq (0,1) \simeq  a\setminus \partial a$$ is obtained by composing the projection on the second factor of the parametrisation of $Z_a$ with the map $s \mapsto \frac{1}2 s+\frac{1}2$.
\end{itemize}
We notice that $F^{-1}(F(Y_v))=Y_v$, that $F(S_a)$ is the middle point of $a$  and that $F^{-1}(F(S_a))=S_a$. Let $\phi : E(M)\longrightarrow E(T)$  be the continuous map induced by $F$.

\begin{lem}\label{lem:homeo}
The map $\phi$ is a homeomorphism.
\end{lem}

\begin{proof}
Let us choose a point $x_v\in C_v$ for each vertex $v$ and, for every edge $a$ joining $v$ to $v'$, an embedded path $\gamma_a\subset C_v\cup Z_a\cup C_{v'}$ joining $x_v$ to $x_{v'}$ in such a way that $\gamma_a\cap Z_a\simeq \{\theta_a\}\times (-1,1)$ for some $\theta_a\in S_a$. We then define a continuous proper map $F': T \longrightarrow M$ so that, for every edge $a\in T$, it is a homeomorphism between $a$ and $\gamma_a$ and  $F'(v)=x_v$.
The map $F'$ satisfy the following two properties:
\begin{itemize}
  \item[$1'$.] $F\circ F'(v)=v$ for every vertex $v\in T$.
  \item[$2'$.] $F\circ F'$ globally preserves each edge $a\in T$. 
\end{itemize}
Let $\phi' : E(T)\longrightarrow E(M)$ be the continuous map induced by $F'$.

By compactness of the spaces of ends it suffices to show that $\phi$ is bijective. By construction $\phi\circ\phi'=id$, hence $\phi$ is surjective. Let $e^* \ne e'^* \in E(M)$ be two ends of $M$ respectively represented by the pre-ends $e=U_1\supset U_2\supset\dots$ and $e'=U'_1\supset U'_2\supset\dots$. We can choose the $U_i$'s and the $U'_i$'s so that they are bounded by spheres in the family $\{S_a\}$. Hence they satisfy, for all $i$, $F^{-1}(F(U_i))=U_i$ and $F^{-1}(F(U'_i))=U'_i$. Since $e\not\simeq e'$ there exist two indices $i$ and $j$ such that $U_i\cap U'_j=\emptyset$. Consequently, $F(U_i) \cap F(U'_j)=\emptyset$. It follows that $\phi (e^*)\ne \phi (e'^*)$. Hence $\phi$ is injective.
\end{proof}

\begin{proof}[Proof of Proposition \ref{prop:invariants}, assertion $1$]
Let us assume that $n_k=n_k(T,f)$ is finite. There exists $n_k$ vertices $v$ of $T$ such that $\widehat Y_v\simeq P_k$. Let $K$ be the union of these $Y_v$.  Then $n_{P_k}(K)=n_k$ and therefore $n_{P_k}(M)\geq n_k$. Let us assume that $n_{P_k}(M)> n_k$. This means, by definition, that there exists a compact submanifold with spherical boundary $L\subset M$ such that $n_k(L)>n_k$. We may assume that $L$ is connected and a union of $Y_v$'s. Consequently, $\widehat L$ can be written as a connected sum of these $\widehat Y_v$ along a subtree $T'$ of $T$ using the restriction $f'$ of $f$ to $T'$. By Milnor's uniqueness there are exactly $n_{P_k}(L)$ components $Y_v$ such that $\widehat Y_v\simeq P_k$. Then $n_k(T', f')\geq n_{P_k}(L)$, and
$$n_k=n_k(T, f)\geq n_k(T', f')\geq n_{P_k}(L) >n_k\,,$$
a contradiction.

Now, if $n_k=n_k(T, f)=+\infty$, for all $n\in \textbf{N}$ there exists some $L_n\subset M$ a union of $Y_v$'s such that $n_{P_k}(L_n)=n$. This shows $n_{P_k}(M)\geq n$ for all $n$,  and hence that $n_{P_k}(M)=+\infty$.
\end{proof}

\begin{proof}[Proof of Proposition \ref{prop:invariants}, assertion $2$]
\mbox{}\\
There remains to prove that $\phi (E_{P_k}(M))=E_k(T,f)$, which is equivalent to
$$n_k(e^*)=+\infty\Longleftrightarrow n_k(\phi(e^*))=+\infty\,.$$
We have:
$$n_k(e^*)=+\infty\Longleftrightarrow n_k(e)=+\infty\Longleftrightarrow n_k(U_i)=+\infty\quad\forall i\,,$$
where $e=U_1\supset U_2\dots$ is a pre-end defining $e^*$. We may choose the $U_i$ saturated by $F$, i.e. satisfying $F^{-1}(F(U_i))=U_i$. Then we set $V_i=F(U_i)$. This is a connected set whose boundary is a finite union of points which are midpoints of edges and images by $F$ of the boundary of $U_i$, which is a finite union of spheres $S_a$.

Let $(K_n)$ be an exhaustion of $M$ by compact sets which are saturated by $F$,  i.e.~$F^{-1}(F(K_n))=K_n$. For $n\in\textbf{N}$, by definition of a pre-end, there exists $i\in\textbf{N}$ such that $U_i\cap K_n=\emptyset$. By the choice of saturated sets we have $V_i\cap F(K_n)=\emptyset$. By definition of $\phi$ the family $V_1\supset V_2\supset \dots$ is a pre-end representing $\phi (e^*)$. 

For $i\in\textbf{N}$ let $T_i$ be the subtree of $T$ whose vertices are the $v$'s such that $Y_v\subset U_i$. The same argument as in the proof of assertion $1$ yields
$$n_{P_k}(U_i)=+\infty\Longleftrightarrow n_k(T_i)=+\infty\,.$$
Therefore, by definition of $V_i$, the vertices of $T$ included in $V_i$ are those of $T_i$ and consequently, for all $i\in\textbf{N}$,
$$n_{P_k}(U_i)=+\infty\Longleftrightarrow n_k(T_i)=+\infty\Longleftrightarrow \vert \{v\in V_i; f(v)=k\}\vert =+\infty \,,$$
which shows that $n_k(\phi (e^*))=+\infty$.
\end{proof}

\begin{proof}[Proof of Proposition \ref{prop:invariants}, assertion $3$]
If the assertion is not true then there exists a prime manifold $P\not\in \mathcal{X}$ and a compact submanifold $K$ of $M$ with spherical boundary such that $n_P (K)\geq 1$. Let $L\subset M$ be a compact connected submanifold with spherical boundary such that $K\subset L$. We may assume that $L$ is a union of $Y_v$'s.  We have that $\widehat L$ is homeomorphic to a connected sum of $\widehat K$ and $\widehat{L\setminus \overset{\circ}K}$. Hence  we have
$n_P(K)\leq n_P(L)=0$. The last equality comes from the hypothesis, the fact that $L$ is a union of $Y_v$'s and thus $\widehat L$ is a connected sum of prime manifolds in $\mathcal X$ and Milnor's uniqueness theorem.

\end{proof}

It follows immediately that $\phi$ is colour-preserving and that for every $k$ the numbers $n_k(T,f)$ and $n_{P_k}(M)$ coincide.

\subsection{Proof of Theorem \ref{theo:construction} }
\label{subsec:theo:construction}

\begin{proof} 
It suffices to prove  the case when $\cX=\{P_1,\ldots,P_k\}$, the general case following easily by adding a finite connected sum of  the $Q_i$'s. In order to follow our convention we add $S^3$ to the family $\cX$ by setting $P_0=S^3$.

Let $\cC$  be the Cantor set which we can identify, by abuse of language, to $\{0,1\}^\NN$. Let $T_0$ be the regular dyadic rooted tree with $E(T)=\cC$. The vertices of $T_0$ are parametrised by finite sequences of 0's and 1's.

We identify $E$ with a closed subset of $\cC$ and let $T$ be the subtree of $T_0$ whose vertices are finite prefixes of elements of $E$.
To each vertex $v$ of $T$ we associate a finite set of integers $F_v$ which is the set of $i\in\{1,\ldots,k\}$ such that $v$ is a prefix of some element of $E_i$.

Form a tree $T'$ by attaching to each vertex $v$ of $T$ a finite tree $T'_v$ with as many vertices as elements of $F_v$. Observe that the obvious retraction from $T'$ to $T$ induces a homeomorphism between $E(T)$ and $E(T')$, so that we can identify these two spaces.

We define a colouring $f$ on $T'$ as follows: for each $v\in T$, set $f(v)=0$, and send the vertices of $T'_v$ bijectively to $F_v$. Let $M$ be the connected sum $M(T,f)$. From Proposition~\ref{prop:invariants} there exists a homeomorphism $\phi : E(M) \to E(T)$ such that for every $i\in\{1,\ldots,k\}$ we have $\phi(E_{P_k}(M))=E_{k}(T', f)=E_k$.
\end{proof}

\section{Proof of Theorem~\ref{thm:tree}}\label{sec:proof of 2.3}

\begin{lem}\label{lem:AB}
Let $M$ be a connected open oriented 3-manifold which is decomposable over the class of all closed 3-manifolds. Let $A \subset M$ be  a connected compact submanifold with spherical boundary such that no component of $M\setminus A$ is bounded.

Then there is a compact connected submanifold $B \subset M$ with spherical boundary  such that:
\begin{enumerate}
\item $A$ is contained in the interior of $B$.
\item Every component $U$ of $M\setminus B$ is unbounded and has connected boundary. 
\item If $C$ is the closure of a component of $B\setminus A$, then $C$ is a punctured $3$-sphere and $\bord C$ has exactly one component contained in $\bord B$.
\end{enumerate}
\end{lem}

\begin{proof}
Let $U_1,\ldots,U_s$ be the components of $M\setminus A$.

For every $1\le i \le s$ we construct a submanifold $C_i$ as follows. If $U_i$ has connected boundary, then we let $C_i$ be a closed collar neighbourhhood of $\bord U_i$ in $U_i$. Otherwise, let $\Sigma_1, \dots ,\Sigma_t$ be the components of $\bord U_i$. For each $1\le j\le t-1$ pick a smooth arc $\gamma_j$ properly embedded in $\overline U_i$ and connecting $\Sigma_j$ to $\Sigma_{j+1}$. Choose a small closed tubular neighbourhood $C_i$ in $\overline{U}$ of $\Sigma_1\cup\dots\cup\Sigma_t\cup\gamma_1\dots\cup\gamma_{t-1}$. Then $C_i$ is a punctured $3$-sphere and its boundary is the union of all $\Sigma_i$'s and some $2$-sphere $S\subset U$.

Doing this for every $i$, we set $B=A\cup \bigcup_{1\le i\le s} C_i$.
\end{proof}

We can now prove Theorem~\ref{thm:tree}.
\begin{proof}
The decomposability hypothesis implies that there exists an exhaustion $(A_n)$ of $M$ by compact connected $3$-submanifolds with spherical boundary, such that each capped-off manifold $\widehat A_n$ is a finite connected sum of members of $\cX$. Without loss of generality, we assume that for every $n$, each component of $M\setminus A_n$ is unbounded. Applying Lemma~\ref{lem:AB} to each $A_n$ and possibly extracting a subsequence and reindexing, we get an exhaustion $\{B_n\}$ by submanifolds with spherical boundary with the following properties:

\begin{enumerate}
\item For each $n$, the submanifold $A_n$ is contained in the interior of $B_n$.
\item For each $n$, every component of $M\setminus B_n$ is unbounded and has connected boundary.
\item If $C$ is the closure of a component of $B_n\setminus A_n$ for some $n$, then $C$ is a punctured $3$-sphere and $\bord C$ has exactly one component contained in $\bord B_n$.
\end{enumerate}

It follows that for every $n$, the capped-off manifold $\widehat B_n$ is a connected sum of $\widehat A_n$ with a finite number of $S^2 \times S^1$ factors.  

Let $\cS$ be the collection of all components of the boundaries of all
$B_n$'s. Then $\cS$ is a locally finite collection of pairwise disjoint, separating, embedded 2-spheres in $M$. The dual graph $T$ of $\cS$ is therefore a tree. It is now easy to attach a finite tree to each vertex of $T$ and get a tree $T'$ such that $M$ is isomorphic to $M(T,f)$ for a suitable colouring $f$,  arguing as in the end of the proof of Theorem~\ref{theo:construction}.
\end{proof}

\section{Proof of the Main Theorem}\label{sec:proof of main}
 
\subsection{Setting up the proof}
Let $\{P_i\}_{0\le k \le r}$ be a finite family of closed oriented 3-manifolds with $P_0=S^3$, $P_1=S^2\times S^1$, and all other $P_k$'s are prime, not isomorphic to $P_1$ and pairwise nonisomorphic. For simplicity we denote the invariants $n_{P_k}(M)$ by $n_k(M)$ and $E_{P_k}(M)$ by $E_k(M)$, when $M$ is a $3$-manifold.

Let $M,M'$ be two open connected oriented manifolds which are decomposable over this family. We suppose that $n_k(M)=n_k(M')$ for all $0<k\le r$, and that there is a colour-preserving homeomorphism 
$\phi : E(M)\longrightarrow E(M')$. We fix once and for all such a homeomorphism. We want to construct an orientation-preserving diffeomorphism (henceforth called an \bydef{isomorphism}) from $M$ to $M'$.

\begin{defi}
Let $A$ be a compact, connected submanifold of $M$ with spherical boundary. We say that a component $U$ of $M\setminus A$  is \bydef{good} if the following properties are satisfied:
\begin{enumerate}\label{list:properties}
 \item The set $U$ is unbounded and has connected boundary.
   \item For every $0<k\le r$, the number $n_k(U)$ is either 0 or $+\infty$.
\end{enumerate}
The set A is said to be  \bydef{good}  if every component of $M\setminus A$ is good  and a spherical exhaustion $\{A_n\}$ is \bydef{good} if each $A_n$ is good.
\end{defi}

Notice that if $A$ is good, then each component of $\bord A$ is separating. Also, for every $k$, if $n_k(M)$ is finite, then $n_k(M\setminus A)=0$. In other words, every prime factor which appears only finitely many times in the prime decomposition of $M$ is contained in $A$.

The idea is inspired by \cite{Rich} (cf~also \cite{Ker}) and consists in constructing two good exhaustions $(A_n)$ of $M$  and $(A'_n)$ of $M'$ together with isomorphisms $\Phi_n : A_n \longrightarrow A'_n$ such that $\Phi_{n+1\vert A_n}=\Phi_n$ and satisfying the following additional property: for every $n$, every component $U$ of $M\setminus A_n$ and every component $U'$ of $M'\setminus A'_n$,  if $\Phi_n (\partial U)=\partial U'$ then $\phi (U^*)=U'^*$. It is then clear that  $\Phi = \lim_{n\to \infty}\Phi_n$ is the desired map.

\subsection{Existence of a good exhaustion} 

\begin{lem}\label{lem:good exhaustion}
The manifold $M$ admits a good exhaustion.
\end{lem}
 
\begin{proof}
This follows by iteration from the following claim:
\begin{claim}
For every bounded subset $A\subset M$, there exists a good submanifold $B$ such that $A$ is contained in the interior of $B$.
\end{claim}

To prove the claim, let $\{K_n\}$ be an ancillary spherical exhaustion of $M$. Without loss of generality we assume that for every $n$, each component of $M\setminus K_n$ is unbounded, and denote these components by
$U_{n,1},\dots, U_{n,s(n)}$. The sets $U^*_{n,i}$ form a basis for the topology of $E(M)$. Any two of these sets are either disjoint or contained in one another.  Every end $e^* \in  E(M)$ is represented by a pre-end $e=V_1\supset V_2\supset\dots$ where the $V_j$ are taken among the $U_{n,i}$'s (see Appendix \ref{sec:ends}).

For any $e^*\in E(M)$ with $e=V_1\supset V_2\supset\dots$ there exists $N\in \textbf{N}$ such that for any $0<k\le r$, we have either
  \begin{itemize}
    \item $n_k(V_i)=0,\quad\forall i\geq N$, (if $e^*\notin E_P(M)$), or
    \item $n_k(V_i)=+\infty$ for all $i$, (if $e^*\in E_P(M)$).
  \end{itemize}
  
For every end $e^*$ we choose such a number $N$ large enough so that $V_N\cap \overline A=\emptyset$, and set $V(e^*):=V_N$. The collection $(V(e^*)^*)_{e^*\in E(M)}$ is an open covering of the compact space $E(M)$. Therefore there exists a finite collection of ends $e^*_1, \dots ,e^*_p$ such that $V(e^*_1)^*\cup \dots \cup V(e^*_p)^*=E(M)$. We set
$$B=M\setminus (V(e^*_1)\cup \dots \cup V(e^*_p))\,.$$

By construction, $B$ is a compact, connected submanifold of $M$ with spherical boundary and contains $A$ in its interior. Moreover, for every component $U$ of $M\setminus B$, the open set $U$ is unbounded and we have $n_k(U)\in \{0,\infty\}$ for all $0<k\le r$.

If for some $U$ the boundary of $U$ is disconnected, argue as in the proof of Theorem~\ref{thm:tree} and add something to $B$ to solve the problem. Then $B$ is good.
\end{proof}

\subsection{Proof of the Main Theorem}

By Lemma~\ref{lem:good exhaustion}, we have good exhaustions $\{B_n\}$ of $M$ and $\{B'_n\}$ of $M'$.

Our goal is to construct by induction two good exhaustions $\{A_n\}$ of $M$ and $\{A'_n\}$ of $M'$ together with a sequence of isomorphisms $\Phi_n:A_n\to A'_n$ with the following properties:

\begin{enumerate}
\item For every $n$, the isomorphism $\Phi_{n+1}$ is an extension of $\Phi_n$.
\item For every $p\le n$, every component $U$ of $M\setminus A_p$ and every component $U'$ of $M'\setminus A'_p$, if $\Phi_p(\bord U)=\bord U'$ then $\phi(U^*)=(U')^*$.
\end{enumerate}

We denote by $(H_n)$ the $n$-th induction step.

The first step of the induction is to show that there exist good subsets $A_0\subset M$ and $A'_0\subset M'$ which are isomorphic and satisfying  assertion 2 above. Since the argument is similar to part of the induction step, we postpone it until the end of the proof.

We do the construction for $n$ even and for $n$ odd we exchange the roles played by $M$ and $M'$.  Therefore, for $n$ even, we assume that the construction of $A_n$, $A'_n$ and $\Phi_n$ has been done and we proceed to the next step. We set $A'_{n+1}=B'_m$, for $m$ large enough so that $\overset{\circ}{A'_{n+1}} \supset A'_n$. Our goal is now to define $A_{n+1}$ together with $\Phi_{n+1} : A_{n+1}\to  A'_{n+1}$ satisfying properties $1$ and $2$.  The first step is to prove that  for $m$ large enough $B_m$ contains a submanifold isomorphic to $A'_{n+1}$.

For a connected component $S$ of  $\partial A_n$, we denote by $U_S$ the component of $M \setminus A_n$ bounded by $S$, and by $U'_{S}$ the component of $M' \setminus A'_n$ bounded by $S':=\Phi_n(S)$. The map $U_S \mapsto U'_{S}$ is a bijection between the set of components of $M \setminus A_n$ and the set of components of $M' \setminus A'_n$. Similarly, for every $m$ large enough so that $\overset{\circ}{B_m } \supset A_n$, we denote by 
$C_{m,S}$ the component of $B_m \setminus \overset{\circ}{A_n}$ such that $C_{m,S} \cap A_n=S$, that is $C_{m,S}=\overline{U_S} \cap B_m$.  Similarly, $C'_{S}$ is  the component of $A'_{n+1} \setminus \overset{\circ}{A'_n}$ such that $C'_{S} \cap A'_n=S'$, that is $C'_{S}=\overline{U'_{S}} \cap A'_{n+1}$.  The map $C_{m,S} \mapsto C'_{S}$ is again a bijection  between the set of components of $B_m \setminus \overset{\circ}{A_n}$ and the set of components of $A'_{n+1} \setminus \overset{\circ}{A'_n}$.

\begin{lem} There exists $m_0 \in \NN$ such that for any integer $m \geq m_0$, for any  $0<k\le r$,  for any connected component $S$ of $\partial A_n$, we have
$$n_k(C_{m,S})\geq n_k(C'_{S})\,.$$
\end{lem}

\begin{proof}
Let $k$ and $S$ be such that  $n_k(C'_{S}) >0$. Using the inclusion $C'_{S} \subset \overline{U'_{S}}$ and the fact that $A'_n$ is good, it follows that $ n_{U'_{S}}(P)=+\infty$, which in turn implies that  $(U'_S)^*\cap E_k(M')\neq\emptyset$.

By the induction assumption $(H_n)$, the corresponding component $U_S$ of $M\setminus A_n$ satisfies $\phi (U_S^*)=(U'_{S})^*$ and, since we also have $\phi(E_k(M))=E_k(M')$, we deduce that  $U_S^*\cap E_k(M)\neq\emptyset$, which implies  $ n_k(U_S)=+\infty$.
Since $C_{m,S}=\overline{U_S}\cap B_m$,  we have $n_k(U_S)=\lim_{m \to \infty} n_k(C_{m,S})$.
Hence, for all $m$ large enough depending on $S$ and $k$,  we have $n_k(C_{m,S}) \geq n_k(C'_{S})$.

Since there are only finitely many $S$ and $k$, we can choose $m$ large enough so that this inequality holds for all of them.
\end{proof}

Next, for $m$ large enough, we will define $A_{n+1}$ by adding to $A_n$ a piece of $C_{m,S}$ for each $S$. Let $(U_{i,m})$ be the collection of connected components of  $M\setminus B_m$ and $(V'_j)$ the collection of connected components of $M'\setminus A'_{n+1}$.

\begin{lem} There exists $m_1\in \NN$ such that for every integer $m \geq m_1$, for each component $U_{i,m}$ of $M\setminus B_m$, there exists a component $V'_j$ of $M'\setminus A'_{n+1}$ such that $\phi (U^*_{i,m})\subset V'^*_j$.
\end{lem}

\begin{proof}
The collection of open sets $(U^*_{i,m})$, for $m\geq m_0$, is a basis for the topology of $E(M)$. The  sets $V'^*_j$ form an open cover of $E(M')$, so the sets $\phi^{-1}(V'^*_j)$ form an open cover of $E(M)$.
Let $e^*\in E(M)$. There exists $j$ such that $e^*\in\phi^{-1}(V'^*_j)$. Since $(U^*_{i,m})$ is a basis for the topology, there exist $i$ and $n\geq m_0$ such that
$$e^*\in U^*_{i,n}\subset \phi^{-1}(V'^*_j)\,.$$
We choose such a $U^*_{i,n}$ and call it $U(e^*)$. Since $E(M)$ is compact, from the collection $(U(e^*))_{e^*\in E(M)}$ we extract a finite covering of $E(M)$,
$$E(M)\subset U^*_{i_1,n_1}\cup \dots \cup U^*_{i_k,n_k}\,.$$
Now, we set $m_1=\max_j\{n_j\}$. For every $m\geq m_1$, each $U_{i,m}$ is included in one $U_{i_s,n_s}$ hence, for each $m\geq m_1$ and for each component $U_{i,m}$ of $M\setminus B_m$, there exists a component $V'_j$ of $M'\setminus A'_{n+1}$ such that $\phi (U^*_{i,m})\subset V'^*_j$.  
\end{proof}

%Let us consider a component $S$ of $\partial A_n$ and recall that $C'_{S'}$ is the component of $A'_{n+1}\setminus\overset{\circ}{A'_n}$  whose intersection with $A'_n$ is the sphere $S'=\Phi_n(S)$. Thus $\partial C'_{S'}$ consists of finitely many $2$-spheres $S', S'_1, \dots,S'_\ell$ where $S'_j\subset \partial A'_{n+1}$, each separating $M'$. We relabel by $V'_j$ the component of $M'\setminus A'_{n+1}$ bounded by $S'_j$.  Notice that $\{V'^*_1,\ldots, V'^*_\ell\}$ is  a  partition of $(U'_{S'})^\ast$. 
% hence $\{\phi^{-1}(V'^*_1),\ldots, \phi^{-1}(V'^*_\ell\})\}$ is a partition of $U_S^*$. 
%Our goal is now to find a submanifold $C_S \subset C_{m,S}$ isomorphic to $C'_{S'}$, with among properties that its boundary spheres contained in $U_S$ induces a partition of $U_S^\ast$ sended by $\phi$ to the above one.

\begin{lem}\label{lem:extension}
For each integer $m \geq m_1$, for every $2$-sphere $S$, there exists a submanifold $C_S \subset C_{m,S}$ with spherical boundary such that the following properties hold:
  \begin{enumerate}[\indent (1)]
    \item Each boundary sphere of $\partial C_S$ separates $M$.
    \item There exists an isomorphism $\Phi_{n+1, S} : C_S \to C'_{S}$ whose restriction to $S$ coincides with $\Phi_n$.
    \item Any component $U$ of $M\setminus C_S$  not bounded by $S$  is good. Furthermore, if $V'$ is the component of $M'\setminus C'_{S}$ bounded by $\Phi_{n+1, S} (\partial U)$, then 
    $\phi (U^*)=V'^*$.
  \end{enumerate}
\end{lem}

Once this lemma is proved,  for $m=\max\{m_0,m_1\}$,  we  define $A_{n+1}=A_n\bigcup_S C_S$ and we extend $\Phi_n$ into $\Phi_{n+1} : A_{n+1} \to A'_{n+1}$ by setting $\Phi_{n+1}=\Phi_{n+1, S}$ on $C_S$. 

\begin{proof}[Proof of Lemma~\ref{lem:extension}]
Let us consider a component $S$ of $\partial A_n$ and recall that $C'_{S}$ is the component of $A'_{n+1}\setminus\overset{\circ}{A'_n}$  whose intersection with $A'_n$ is the sphere $S'=\Phi_n(S)$. Thus $\partial C'_{S}$ consists of finitely many $2$-spheres $S', S'_1, \dots,S'_\ell$ where $S'_j\subset \partial A'_{n+1}$, each separating $M'$. We relabel by $V'_j$ the component of $M'\setminus A'_{n+1}$ bounded by $S'_j$.  Notice that $\{V'^*_1,\ldots, V'^*_\ell\}$ is  a  partition of $(U'_{S})^\ast$. 

Recall that $C_{m,S}$ is the component of $B_m \setminus \overset{\circ}{A_n}$ such that $C_{m,S} \cap A_n=S$, and that $\phi(U_S^\ast)=(U'_{S})^\ast$.  We  partition $\partial C_{m,S}$ into $S \cup \Gamma_1\cup\cdots\cup\Gamma_\ell$ where each $\Gamma_j$ is a union of $2$-spheres of $\partial B_m$, and label the $\Gamma_j$'s in such a way that  if $\partial \overline{U_{i,m}}\subset \Gamma_j$ for some $i$ and $j$, then $\phi (U^*_{i,m})\subset V'^*_j\,$. It follows that if $V_j$ denotes the union of the $U_{i,m}$'s such that $\partial \overline{U_{i,m}}\subset \Gamma_j$, then $\{V_1^\ast, \ldots,V_\ell^\ast\}$ is a partition of $U_S^\ast$ and $\phi(V_j^\ast)= V'^*_j\,$.

We now consider the closed oriented $3$-manifolds $\widehat C$ and $\widehat C'$ obtained from $C_{m,S}$ and $C'_{S}$ by gluing $3$-balls to their boundary components. They both have a Kneser-Milnor decomposition in elements of $\mathcal X$:
$$\widehat C\simeq \overbrace{P_1\#\dots\# P_1}^{n_1}\# \overbrace{P_2\#\dots\# P_2 }^{n_2}\dots \#\overbrace{P_r \dots \# P_r}^{n_r}\# P_{0}$$
with $n_i=n_i(\widehat C)$, and
$$\widehat C'\simeq \overbrace{P_1\#\dots\# P_1}^{n'_1}\# \overbrace{P_2\#\dots\# P_2 }^{n'_2}\dots \#\overbrace{P_r \dots \# P_r}^{n'_r}\#P_{0} \,.$$
with $n'_i=n_i(\widehat C')$.

From the previous arguments, $n_k\geq n'_k\geq 0$ for each $0<k\le r$. By construction 
$C'_{S}$ is a submanifold of $ \widehat{C'}$ and  $\widehat C' \setminus \overset{\circ}{C'_{S}}$ is a collection of $3$-balls $D',D'_1,\ldots,D'_\ell$ bounded by $S',S'_1,\ldots,S'_\ell$, each sphere separating $ \widehat{C'}$.  Consider now the following connected sum 
$$\widehat{C'}\#N'_1\#\dots\#N'_\ell\,.$$ 
where each $N'_j$ is a closed manifold chosen below and is attached to $D'_j$. We claim that we can choose $N'_j$ so that 
\begin{enumerate}[\indent (a)]
\item $\widehat{C'} \#N'_1\#\dots\#N'_\ell \approx \widehat C$,
\item  For all $0<k\le r$, if $n_k(N'_j)>0$ then $n_k(V'_j)=\infty$.
\end{enumerate}

The manifolds $N'_j$ are defined as connected sums of members of $\cX$. The fact that we can find manifolds $N'_j$ satisfying (a) follows from the above observation that each prime component of the 
Kneser-Milnor decomposition of $\widehat{C'}$ appears at least as many times as in the prime decomposition of $\widehat C$. It then suffices to divide the supernumerary prime components of $\widehat C$ (if any) into the $N'_j$'s. For (b) we first remark that, for $0<k\le r$, 
$$n_k(\widehat{C}) >0 \Longrightarrow n_k(U_S) > 0 \Longrightarrow n_k(U_S)= + \infty \Longrightarrow n_k(U'_{S})= + \infty\,,$$
the last implication following from $(H_n)$,  the previous one from goodness.

Finally, $$n_k(U'_{S'})= + \infty\Longrightarrow \exists j,\quad n_k(V'_j)=+\infty\,.$$  

Summarising, for any prime component $P_k$ of $\widehat C$ there exists $j \in \{1,\ldots,\ell\}$ such that $n_k(V'_j)=+\infty$. For each supernumerary $P_k$, we select $j=j(P_k)$ such that $n_k(V'_j)=+\infty$. We define $N'_j$ to be the connected sum of $S^3$ with the supernumerary manifolds $P_k$ for which $j(P_k)=j$. It follows that the $N'_j$'s satisfy (b).

%Now, for a fixed $j$, let us do the connected sum of all those elements $P\in\mathcal X$ that appear in the above decomposition of $\widehat C$ and not in the decomposition of $\widehat C'$ and that, furthermore satisfy $n_{V'_j}(P) >0$. With the above notations they are among the $(n_i-n'_i)$ (if non zero) $P_i$'s for $1\leq i\leq k$, or among the $P_i$'s for $i\geq k+1$. We denote this connected sum by $N'_j$. If, for some 
%$k$, $n_{V'_k}(P) = 0$ for all $P$'s in the decomposition of $\widehat C$, we set $N'_k=S^3$. We then define $\widetilde C'$ by,
%$$\widetilde C'= \widehat C'\#N'_1\#\dots\#N'_p\,.$$ 
%Here $N'_j$ is connected to $C'$ by removing to it a $3$-ball and gluing the resulting manifold with boundary at the $2$-sphere $S'_j$. From the uniqueness of Kneser-Milnor decomposition, $\widetilde C'$ is isomorphic to $\widehat C$. 

Let us denote by $f:\widehat C\to\widehat{C'}\#N'_1\#\dots\#N'_\ell$ an isomorphism. Let $D, D_1, \dots , D_s \subset \widehat C$ be the compact $3$-balls such that $\widehat C = C_{m,S} \cup D \cup D_1 \cup \cdots \cup D_s$ and $\partial C_{m,S}$ consists of the boundary of these disks, where $S = \partial D$. 

The images $f(D_j)$ are disjoint compact $3$-balls in $\widehat{C'}\#N'_1\#\dots\#N'_\ell$.  These balls can be isotoped anywhere in $\widehat{C'}\#N'_1\#\dots\#N'_\ell$. Hence, up to composing $f$ with finitely many diffeomorphisms we may assume that,
$$f(D)=D'\,, \,(\textrm{thus } f(S)=S')$$
and that 
$$\partial D_i\subset \Gamma_j\Longrightarrow f(D_i)\subset N'_j\,.$$
Then $f(C_{m,S})\supset C'_S$.
We then define,
$$C_S:=f^{-1}(C'_{S})\subset C_{m,S}\,,\quad  S_j:= f^{-1}(S'_j)\subset \partial C_S\subset C_{m,S}.$$

Recall that $S'_j$ separates $M'$, hence also separates $\widehat{C'}\#N'_1\#\dots\#N'_\ell$. It follows that 
 $S_j$ separates $\widehat C$, hence also $C_{m,S}=\widehat C \setminus \Int(D \cup D_1 \cup \ldots \cup D_s)$. As each boundary sphere of $C_{m,S}$  separates $M$, it follows that each $S_j$ separates $M$, which implies  property (1) of Proposition \ref{lem:extension}. To verify property (2)  let us notice that we may deform $f_{\vert C_S}$ near $S$ into a diffeomorphism $\Phi_{n+1,S} : C_S \to C'_{S'}$ that satisfies 
$${\Phi_{n+1,S}}_{\vert S}=\Phi_{n\vert S}:S \longrightarrow S'= \Phi_n(S)\,.$$
Indeed, the space of orientation-preserving diffeomorphisms of $S^2$ is path-connected and this follows, for example, from  \cite{smale:diff}. To check property (3) observe that the connected component $U_j$ of  $M \setminus C_S$ bounded by $S_j$ contains the union of connected components $V_j$ of $M \setminus B_m$ bounded by $\Gamma_j$ and that $U_j \setminus V_j \subset B_m$ is bounded. Therefore ${U_j}^\ast={V_j}^\ast$, from  which we infer that  $\phi({U_j}^\ast)=\phi({V_j}^\ast)={V'_j}^\ast$ as required.
This concludes the proof of Lemma \ref{lem:extension}.
As explained before, extending $\Phi_n$ by $\Phi_{n+1,S}$ in order to define $\Phi_{n+1}$ concludes the proof of the induction.

We only need now to explain why assertion $(H_0)$ is true. Let us set $A'_0:=B'_0$. As $(B_n)$ is an exhaustion of $M$ and $n_k(M)=n_k(M')$, for all $0<k\leq r$, one can find $m_0\in\textbf{M}$ such that, for all $0<k\leq r$ and for all $m\geq m_0$,
$$n_k(B_m)\geq n_k(A'_0)\,.$$
We denote by $S'_1,\dots,S'_l$ the components of $\partial A'_0$ and by $V_j$ the components of $M'\setminus A'_0$ bounded by $S'_j$, for $j=1,\dots,l$. As in Lemma \ref{lem:extension}, there exists $m_1$ such that, for all $m\geq m_1$, if $(U_{i,m})$ denotes the components of $M\setminus B_m$, then, for all $i$ and $j$,
$$\phi (U_{i,m})^*\subset (V'_j)^*\,.$$
Let us fix such $m\geq \sup\{m_0,m_1\}$. We can now argue as in Lemma \ref{lem:extension} with $A'_0$ instead of $C'_S$ to define a submanifold $A_0\subset B_m$ and an isomorphism $\Phi_0 : A_0\longrightarrow A'_0$ such that $S_j:=\Phi_0^{-1}(S'_j)$ separates $M$ and such that  the component $U_j$ of $M\setminus A_0$ bounded by $S_j$ is good and satisfies
$$\Phi(U^*_j)=(V'_j)^*\,.$$
\end{proof}

\section{Final remark}\label{sec:conclusion}

The above results exhibit a subtle interaction between the combinatorial structure of trees and the topological structure of open decomposable $3$-manifolds. It is an interesting question to make this relation more precise and in particular to try to find the more efficient tree representing, up to homeomorphism, a given decomposable $3$-manifold. If the manifold is closed, then a linear, finite tree is a good option but if it is open we do not have, at the moment, any hint regarding an optimal description.

\appendix

\section{Space of ends}\label{sec:ends}

%We are only interested in the case where $X$ is a locally finite graph
% or a manifold, so we give a definition adapted to this case.
%For a more general definition (?) see e.g. \cite{dickman-mccoy}. 

Let $X$ be a connected, locally compact, metrisable, separable topological space. We say that a subset of $X$ is \bydef{bounded} if it is contained in some compact subset. Otherwise it is \bydef{unbounded}. A \emph{pre-end} of $X$ is a decreasing sequence $$e=U_1\supset U_2 \supset \dots\,,$$ of open subsets of $X$ such that:
\begin{enumerate}
\item Every $U_i$ is connected, unbounded, and has compact boundary, and
\item For every bounded $A\subset X$, the set $A\cap U_i$ is eventually empty.
\end{enumerate}
 
Let $e=U_1\supset U_2 \supset \dots$ and $e'=U'_1\supset U'_2 \supset \dots$ be pre-ends of $X$. We say that $e$ is \bydef{contained in} $e'$, and write $e \subset e'$ if for every $n'\in \mathbf{N}$ there exists $n \in \NN$ such that $U_n\subset U'_{n'}$. We define an equivalence relation by letting $e \sim e'$ if $e \subset e'$ and $e'\subset e$. The equivalence class of $e$ is denoted by $e^*$. We call  \bydef{end} of $X$ such an equivalence class, and denote by $E(X)$ the set of ends of $X$.

For every open subset $U\in X$ with compact boundary we set 
$$U^*:=\{ e^* \in E(X) ; e=U_1\supset U_2 \supset \dots\,, \textrm{ and } U_i\subset U\quad \forall i \textrm{ large enough}\}\,.$$

Those sets form a basis for a topology on $E(X)$. Hereafter the set $E(X)$ is endowed with this topology and called the \bydef{space of ends} of $X$. Freudenthal~\cite{freudenthal} proved that this space is compact, metrisable and totally disconnected. Hence it embeds into the Cantor set (see e.g. \cite[7.B and 7.D]{kechkris}). 

The following technical result is well-known. Having being unable to find a reference, we provide a proof.
\begin{lem}
Let $M$ be a connected open manifold and $(K_n)$ be an exhaustion of $M$. Let $\{U_{n,i}\}$ be the collection of all unbounded components with compact boundary of all the sets $M\setminus K_n$. Then the topology on $E(M)$ is generated by the $(U_{i,n}^\ast)$'s.
\end{lem}
\begin{proof}
First we show that any end of $M$ can be represented by $U_1 \supset U_2 \supset \ldots$ where $U_j  \in\{U_{n,i}\}$. 

Let $e=U'_1\supset U'_2\supset \dots$ be a pre-end of $M$. By compactness of $\partial U'_j$, one has $\partial U'_j\subset \overset{\circ}{K}_n$ for $n$ large enough. Fix an increasing sequence $n(j)$ such that $\partial U'_j\subset \overset{\circ}{K}_{n(j)}$ for all $j$. As $U'_j$ is unbounded, it meets $M\setminus K_{n(j)}$, hence some component $U_{n(j),i}$. Notice that any component 
 $U_{n(j),i}$ of $M\setminus K_{n(j)}$ which meets $U'_j$ is contained in $U'_j$:  if not, there exists a path $\gamma \subset U_{n(j),i}$ going from $U'_j$ to its complement, hence crossing $\partial U'_j \subset  \overset{\circ}{K}_{n(j)} \subset M \setminus U_{n(j),i}$, a contradiction. Now, for each $j$ choose $m(j)$ large enough so that $K_{n(j)} \cap U'_{m(j)} = \emptyset$. Then select the component $U_{n(j),i(j)}$ of $M \setminus K_{n(j)}$ which meets $U'_{m(j)}$. By connectedness we have 
$$U'_j\supset U_{n(j),i(j)} \supset U'_{m(j)}.$$

Setting $U_j := U_{n(j),i(j)}$, the pre-end $U_1 \supset U_2 \supset \ldots$ is equivalent to $e$. The family $\{U_{i,n}^\ast\}$ is an open cover of $E(X)$. Any two $U,V$ in  $\{U_{n,i}\}$  satisfy $U \cap V=\emptyset$ or $U \subset V$ or $U \supset V$. The same property holds for $U^\ast,V^\ast$ in $\{U_{n,i}^\ast\}$ and implies that  $\{U_{n,i}^\ast\}$ is a basis for some topology. 

 Let $U^*$ be an element of a  basis for the topology of $E(X)$, and let $e^*\in U^*$. We can represent $e^*$ by $e=U_1\supset U_2\dots$ where $U_j\in \{U_{n,i}\}$ and by definition there exists $j$ such that $U\supset U_j$. Then $U^*\supset U_j^*\ni e^*$. This shows that $\{U^*_{n,i}\}$ is a basis for the topology of $E(X)$.
\end{proof}

Let $e$ be a pre-end of $X$ and $U$ be an open subset of $X$. We say that $U$  is a \bydef{neighbourhood} of $e^*$ if $e^*\in U^*$. Thus a pre-end is a non-increasing sequence of neighbourhoods of the associated end.
\end{spacing}

\bibliographystyle{alpha}
\bibliography{biblio}

\end{document}